\documentclass[11pt]{amsart}

\usepackage{amsmath}
\usepackage{amssymb}
\usepackage{graphics}

\theoremstyle{plain}

\newtheorem{theorem}{Theorem}
\newtheorem{corollary}{Corollary}[theorem] 

\newtheorem{lemmaX}{Lemma}

\newtheorem{example}{Example}
\newtheorem{remark}{Remark}
\theoremstyle{definition}
\newtheorem{definition}{Definition}

\title[Convergence of remote projections]{Convergence of remote projections \\ onto convex sets}

\author[P. A. Borodin]{Petr A. Borodin}
\address{Department of Mechanics and
Mathematics, Moscow State University, Moscow 119991, Russia}
\address{and}
\address{Moscow Center for Fundamental and Applied Mathematics}
\email{pborodin@inbox.ru}
\author[E. Kopeck\'a]{ Eva Kopeck\'a}
\address{Universit\"at Innsbruck \\
Department of Mathematics\\
 Technikerstra\ss e 13  \\
6020 Innsbruck, Austria}
\email {eva.kopecka@uibk.ac.at}



\thanks{The first author was supported by the Russian Science Foundation (grant no. 23-71-30001) in Moscow State University.}

\begin{document}

\begin{abstract}
Let $\{C_{\alpha}\}_{\alpha\in \Omega}$ be a family of closed and convex  sets in  a Hilbert space $H$, having a nonempty intersection $C$. We consider a sequence $\{x_n\}$  of remote projections onto  them. This means,
$x_0\in H$,  and $x_{n+1}$   is  the projection of $x_n$ onto
such a set $C_{\alpha(n)}$ that the ratio of the distances from $x_n$ to this set and to any other  set from the family is at least  $t_n\in [0,1]$. 
We study properties of the weakness parameters  $t_n$ and of  the sets $C_\alpha$  which ensure the  norm or weak convergence of the sequence $\{x_n\}$   to a point in $C$. 
We show that condition (\ref{teml_cond})  is  necessary and sufficient for the norm convergence of $x_n$ to a point in $C$ for any starting element and any family of closed, convex, and symmetric sets $C_\alpha$.
This generalizes a result of Temlyakov who introduced (T) in the context of  greedy approximation theory. 
We  give examples explaining to what extent the symmetry condition on the sets $C_{\alpha}$ can be dropped. 
Condition (T)  is stronger than $\sum t_n^2=\infty$ and weaker than $\sum t_n/n=\infty$.
The condition $\sum t_n^2=\infty$  turns out to be necessary and sufficient for the sequence $\{x_n\}$ to have a partial weak limit in $C$ for any family of closed and convex sets $C_\alpha$ and any starting element.
\end{abstract}

\maketitle

\section*{Introduction}

In the entire paper $H$ is a real Hilbert space; its
 scalar product we denote by $\langle\,\cdot\,,\cdot\,\rangle$ and the corresponding  norm by $|\,\cdot\,|$.

Let $\{C_\alpha\}_{\alpha\in \Omega}$ be a family of closed and convex sets in  $H$, $|\Omega|\geq 2$, so that  $C=\bigcap_{\alpha\in \Omega}C_\alpha\not=\emptyset$.  Let $P_\alpha$ denote the metric projection onto $C_\alpha$. Sometimes we also denote the metric projection onto a closed convex set $A$  by $P_A$.

A fixed sequence $\{\alpha(n)\}\subset \Omega$ and a starting element $x_0\in H$ generate the sequence
$x_{n+1}=P_{\alpha(n)} x_n$, $n=0,1,2,\dots$, of consecutive projections  that can be examined for convergence.

In the case when $C_\alpha$ are closed linear subspaces of $H$ and $\Omega$ is a finite set the convergence properties of the sequence $\{x_n\}$ are well understood. If the sequence of indices $\{\alpha(n)\}$ is cyclic, that is, $\Omega=\{0, 1,\dots, K-1\}$ and $\alpha(n)\equiv n\,(\hbox{mod}\,K)$, then $\{x_n\}$ converges in norm to a point of $C$~\cite{N}, \cite{Ha}. The rate of convergence depending on the position of the subspaces and  of the initial point can be estimated (see~\cite{BK20}, \cite{RZ}, and the bibliography there). Already for three closed linear subspaces divergence might occur if no extra information about the sequence of indices, or about the geometry  of the  subspaces is known  (see~\cite{KM}, \cite{KP}, and \cite{K20}). The sequence $\{x_n\}$, however, always converges weakly to a point in $C$ if $\Omega$ is finite, each $\alpha\in \Omega$ occurs infinitely many times in the sequence $\{\alpha(n)\}$ and the sets $C_\alpha$ are closed linear subspaces~\cite{AA}.

In the lack of linearity, when the sets $C_\alpha$ are just closed and convex, the situation is different.
Already for $|\Omega|=2$ the sequence $\{x_n\}$ might diverge in norm  (\cite{H}, see also \cite{K} and \cite{MR}).
Weak convergence is known only under additional conditions~\cite{BK22}: when $|\Omega|\leq 3$ \cite{DR}, or when $|\Omega|<\infty$ and the indices $\{\alpha(n)\}$ are cyclic~\cite{Bre}, or when $|\Omega|<\infty$ and the sets are ``somewhat symmetric''   \cite{DR}, or when $\Omega$ is arbitrary but each successive projection takes place on the farthest set~\cite{Bre}.  In this article we develop the last plot.

For any starting element $x_0\in H$, we consider the sequence of {\em remote} projections
\begin{equation}\label{rem}
x_{n+1}=P_{\alpha(n)}x_n, \qquad n=0,1,2,\dots,
\end{equation}
where $\alpha(n)\in \Omega$ is chosen so that
$$
{\rm dist\,}(x_n,C_{\alpha(n)})\ge t_n\sup_{\alpha} {\rm dist\,}(x_n,C_\alpha),
$$
and $t_n\in [0,1]$ are prescribed {\em weakness parameters}. If there is at least one $n\in {\mathbb N}$    with $t_n=1$, that is, when the $n$th projection is the  {\em remotest}, we require the maximum  $\max_{\alpha} {\rm dist\,}(x,C_\alpha)$  to be attained for each $x\in H$. Note that any sequence of consecutive projections onto the family $\{C_\alpha\}$ can be regarded as a sequence of remote projections with some,   possibly very small, $t_n$'s.

We prove a convergence result for the remote projections (\ref{rem}). If all the convex sets  $C_\alpha$ are closed subspaces of codimension one, this convergence theorem is already known within the greedy approximation theory~\cite{T}.

We recall the corresponding definitions.
A subset $D$ of the unit sphere $S(H)=\{s\in H: |s|=1\}$ is called a {\it dictionary} if $\overline{{\rm span\,}} D=H$.
For any dictionary $D\subset S(H)$, any sequence $\{t_n\}$ in $[0,1]$ of weakness parameters, and any $x_0\in H$, the {\em Weak Greedy
Algorithm} (WGA) generates a sequence $x_n$ defined inductively
by
\begin{equation}\label{greal}
x_{n+1}=x_n-\langle x_n,g_n\rangle g_n, \qquad n=0,1,2,\dots,
\end{equation}
where the element $g_n\in D$ is such that
$$
|\langle x_n,g_n\rangle |\ge t_n\sup\{|\langle x_n,g\rangle|: g\in D\}.
$$
The weakness parameters were introduced  by Temlyakov in \cite{T00}.
In the case when  $t_n\equiv 1$ of the {\em Pure Greedy Algorithm}, or even if $t_n=1$ for at least one  $n$,  we  require the maximum $\max \{|\langle x_n,g\rangle|: g\in D\}$ to be attained for each $x\in H$. We say  that the WGA converges if $|x_n|\to 0$.

Clearly,   the WGA coincides with the process of remote projecting onto the family of hyperplanes $\{g^\perp: g\in D\}$ orthogonal to the dictionary elements. Since $D$ is spanning, the origin is the only point in the intersection  of these hyperplanes. The other way round it works as well.

\begin{remark}\label{WGA}
Let $\{C_\alpha\}$ be a family of closed linear subspaces with $\bigcap_\alpha C_\alpha=\{0\}$. Then $\bigcup_\alpha C_\alpha^{\perp}$ is spanning and the remote projections (\ref{rem}) correspond to the WGA 
with respect to the dictionary $\bigcup_\alpha C_\alpha^\perp\cap S(H)$ with the same sequence $\{t_n\}$ of weakness parameters.
\end{remark}

According to Temlyakov \cite{T02} the condition
\begin{equation}\tag{T}\label{teml_cond}
\forall \{a_\nu\}\in \ell_2 \mbox{ with }  a_\nu\ge 0:\, \liminf_{m\to \infty} \frac{a_m}{t_m}\sum_{\nu=1}^m a_\nu=0
\end{equation}
on the sequence $\{t_m\}$ is necessary and sufficient for the convergence of all realizations of the WGA with the weakness sequence $\{t_m\}$
for each $x_0\in H$ and all dictionaries $D\subset S(H)$.

The condition (\ref{teml_cond}) is  rather subtle. For instance,
$$
\sum \frac{t_m}{m}=\infty \, \Rightarrow\,  {\rm (\ref{teml_cond})}\,  \Rightarrow\,  \sum t_m^2=\infty,
$$
but none of the implications can be reversed.

In Section~\ref{norm} we generalize the convergence theorem of \cite{T02} to the case of remote projections onto a family of closed convex sets $C_\alpha$ that are uniformly quasi-symmetric with respect to their common point; see Theorem~\ref{theorem1} below.
Our proof partially leans on
 Temlyakov's paper~\cite{T02} and  on  the proof of Jones' theorem on the convergence of the Pure Greedy Algorithm in Hilbert space~\cite{J}, \cite[Ch.\,2]{T}. The quasi-symmetry condition is essential in view of Hundal's example~\cite{H}, \cite{K}.
 In Corollary~\ref{almostperiodic} we   show that  cyclic projections onto finitely many closed, convex and quasi-symmetric sets converge in norm. This generalizes a result from \cite{BBR}.

 In Section~\ref{sym} we discuss different versions of the quasi-symmetry condition, and show that the uniform quasi-symmetry is   essential in Theorem~\ref{theorem1}.

Section~\ref{weak} is devoted to the weak convergence of remote projections. In Theorem~\ref{theorem2} we  give a condition on $\{t_n\}$ sufficient for weak convergence to a point in $C$. 
As Corollary~\ref{walmostperiodic} we  get a result of \cite{DR}:    quasi-periodic projections onto finitely many closed and convex   sets converge  weakly to a point in their intersection.
In Theorem~\ref{theorem3} we give  another condition on $\{t_n\}$ that is necessary and sufficient for all sequences of remote projections to have a partial weak limit in $C$. We construct an example showing  certain sharpness of these theorems.

\section{Norm convergence}\label{norm}

In this section we show when remote projections onto a family of closed and convex sets converge. A symmetry assumption on the sets is needed;  we define this weakened symmetry below.
In Section~\ref{sym} we will show to what extent this symmetry condition is necessary. We will also compare it to another weakened symmetry condition.

In what follows $B(a,r)$ denotes the closed ball with center $a$ and radius $r>0$.

\begin{definition}\label{quasym}
 Let $C$ and $C_\alpha$, $\alpha\in \Omega$, be   closed convex sets in a Banach space $X$, all containing the origin.
\begin{enumerate}
\item[(i)] We call the set $C$ {\em quasi-symmetric}, if \\
$\forall r>0 \ \exists\, \theta=\theta(r)\in(0,1]:  \   x\in C\cap B(0,r)\ \Rightarrow
    -\theta x \in C$.
\item[(ii)] We say that the family of sets $\{C_\alpha\}_{\alpha\in \Omega}$ is uniformly quasi-symmetric if \\
$\forall r>0 \ \exists\, \theta=\theta(r)\in(0,1]  \, \forall \alpha:\  x\in C_\alpha\cap B(0,r)
     \Rightarrow -\theta x\in C_\alpha$.
\end{enumerate}
Moreover, we say that $C$ is  {\em quasi-symmetric} with respect to a point $a\in C$ if the set  $(C-a)$ is quasi-symmetric.
Similarly, the family of sets $\{C_\alpha\}_{\alpha\in \Omega}$ is uniformly quasi-symmetric with respect to a point $a\in \bigcap_{\alpha\in \Omega} C_\alpha$ if the family $\{C_\alpha-a\}_{\alpha\in \Omega}$ is uniformly quasi-symmetric.
\end{definition}

In the above definition we   can equivalently write ``$\exists\,r>0$'' instead of ``$\forall\,r>0$''. Indeed, $\theta(r)=\theta_0\min\{1,r_0/r\}$ works  for a given $r>0$, if $\theta_0$  works for  some $r_0>0$.

\begin{theorem}\label{theorem1} 
For a sequence $\{t_n\}_{n=0}^\infty \subset [0,1]$, the following two statements are equivalent:
\begin{enumerate}
\item[(i)] The sequence $\{t_n\}$ satisfies the condition (\ref{teml_cond}).
\item[(ii)]  For any family $\{C_{\alpha}\}_{\alpha\in \Omega}$ of closed and convex sets  in a Hilbert space $H$ which is uniformly quasi-symmetric with respect to a point $a\in C=\bigcap_{\alpha\in \Omega} C_\alpha$ and for any starting element $x_0\in H$ the sequence $\{x_n\}$ of remote projections (\ref{rem})  
converges in norm to a point in  $C$.
\end{enumerate}
\end{theorem}
\begin{proof}
We will prove here  that (i) $\Rightarrow$ (ii).  The implication (ii) $\Rightarrow$ (i) follows from \cite{T02} where Temlyakov shows that the condition  (\ref{teml_cond}) is necessary already when all $C_{\alpha}$'s are hyperplanes.

 1. We can  assume that $a=0$. Let $x_n-x_{n+1}=y_n$, $\angle 0x_{n+1}x_n=\pi/2+\varepsilon_n$. We have $\varepsilon_n \in [0,\pi/2]$; otherwise $x_{n+1}$ is not the nearest point for $x_n$ in the segment $[0,x_{n+1}]$ and hence also in $C_{\alpha(n)}$. Consequently, $|x_n|^2\ge |x_{n+1}|^2+|y_n|^2$, so that the norms $|x_n|$ decrease to some $R\ge 0$. We suppose $R>0$, otherwise $x_n\to 0$. Moreover,
\begin{equation}\label{y_n}
\sum_{n=0}^\infty |y_n|^2\le \sum_{n=0}^\infty (|x_n|^2-|x_{n+1}|^2)<\infty.
\end{equation}
Consequently, since $\{t_m\}$ satisfies (\ref{teml_cond}), we can choose a subsequence $\Lambda\subset {\mathbb N}$ with the property
\begin{equation}\label{Lambda}
b_m:=\frac{|y_m|}{t_m}\sum_{\nu=0}^m |y_\nu|\to 0, \qquad m\to\infty,\,  m\in \Lambda.
\end{equation}

2. Now we prove that
\begin{equation}\label{sin}
\sum_{n=0}^\infty |y_n| \sin \varepsilon_n<\infty.
\end{equation}

By the law of cosines,
\begin{equation}\notag
\begin{split}
|x_n|^2&= |x_{n+1}|^2+|y_n|^2-2|x_{n+1}||y_n|\cos \left(\frac{\pi}{2}+\varepsilon_n\right) \\
&= |x_{n+1}|^2+|y_n|^2+2|x_{n+1}||y_n|\sin \varepsilon_n,
\end{split}
\end{equation}
so that
$$
|y_n|\sin \varepsilon_n=\frac{|x_n|^2- |x_{n+1}|^2-|y_n|^2}{2|x_{n+1}|}\le \frac{|x_n|^2- |x_{n+1}|^2}{2R},
$$
and (\ref{sin}) follows.

3. The vector $y_\nu$ is externally normal to a hyperplane supporting the set $C_{\alpha(\nu)}$ at the point $x_{\nu+1}$. Hence,
\begin{equation}\label{ext_normal}
\langle y_\nu,z-x_{\nu+1}\rangle \le 0
\end{equation}
for any $z\in C_{\alpha(\nu)}$, so that
\begin{equation}\notag
\begin{split}
\langle y_\nu, z\rangle &\le \langle y_\nu, x_{\nu+1} \rangle = |y_\nu| | x_{\nu+1}|\cos \left( \frac{\pi}{2}-\varepsilon_\nu\right) \\
&=|y_\nu| | x_{\nu+1}|\sin \varepsilon_\nu \le |y_\nu| | x_0|\sin \varepsilon_\nu.
\end{split}
\end{equation}

Since $C_{\alpha(\nu)}$ is quasi-symmetric with respect to 0, we get
\begin{equation}\label{z}
|\langle y_\nu, z\rangle |\le \theta^{-1}|y_\nu| | x_0|\sin \varepsilon_\nu, \qquad z\in C_{\alpha(\nu)}\cap B(0,2|x_0|),
\end{equation}
where $\theta=\theta(2|x_0|)$ is from the  definition of uniform quasi-symmetry.

4. Next, for any $m$ and $\nu$,
\begin{equation}\notag
\begin{split}
|y_m|&={\rm dist\,}(x_m,C_{\alpha(m)})\ge t_m\sup_{\alpha} {\rm dist\,}(x_m,C_\alpha)\ge t_m{\rm dist\,}(x_m,C_{\alpha(\nu)}) \\
&=t_m \min_{z\in C_{\alpha(\nu)}} |x_m-z|=t_m \min_{z\in C_{\alpha(\nu)}\cap B(0,2|x_0|)} |x_m-z| \\
&\ge t_m\min_{z\in C_{\alpha(\nu)}\cap B(0,2|x_0|)} |\langle x_m-z, y_\nu/|y_\nu|\rangle| \\
&\ge t_m \left( |\langle x_m, y_\nu/|y_\nu|\rangle|-\max_{z\in C_{\alpha(\nu)}\cap B(0,2|x_0|)} |\langle z, y_\nu/|y_\nu|\rangle|\right) \\
&\ge  t_m \left(|\langle x_m, y_\nu/|y_\nu|\rangle|-\theta^{-1}|x_0|\sin \varepsilon_\nu\right);
\end{split}
\end{equation}
 we have  used (\ref{z}) in the last inequality. The above estimate implies that
\begin{equation}\label{m_nu}
|\langle x_m, y_\nu\rangle |\le \frac{|y_m| |y_\nu|}{t_m} +\theta^{-1} | x_0||y_\nu|\sin \varepsilon_\nu
\end{equation}
for any $m, \nu =0,1,2,\dots$

5. Now we  prove that $\{x_n\}$ is a Cauchy sequence and hence it converges to some $w\in H$. Given any $n,k\in {\mathbb N}$ assume that $m \in \Lambda$ and $m>\max\{n,k\}$.
Since $|x_n-x_k|\le |x_n-x_m|+|x_k-x_m|$, it is enough to show  that
$|x_n-x_m|\to 0$ as $n,m \to \infty$, $n<m$ and $m\in \Lambda$.
We use the identity
$$
|x_n-x_m|^2= |x_n|^2-|x_m|^2-2\langle x_n-x_m, x_m  \rangle.
$$
Since $|x_n|\to R$, we have  $|x_n|^2-|x_m|^2\to 0$ as $n,m\to \infty$. The last term we estimate using (\ref{m_nu}) as follows:
\begin{equation}\notag
\begin{split}
|\langle x_n-x_m, x_m  \rangle|&= \left| \sum_{\nu =n}^{m-1} \langle y_\nu, x_m\rangle\right|\le \sum_{\nu =n}^{m-1} |\langle y_\nu, x_m\rangle| \\
&\le \frac{|y_m| }{t_m}\sum_{\nu =n}^{m-1} |y_\nu| +  \theta^{-1}| x_0| \sum_{\nu =n}^{m-1}  |y_\nu| \sin \varepsilon_\nu.
\end{split}
\end{equation}
The first sum does not exceed $b_m$ by (\ref{Lambda}), so it tends to 0 as $m\to \infty$, $m\in \Lambda$. The second sum tends to 0 as $n,m\to \infty$ in view of (\ref{sin}).

6. Finally  we show that $w=\lim x_n$ is  contained  in all $C_\alpha$'s. If  $w\notin C_\beta$ for some $\beta$, then
$$
{\rm dist\,}(x_n,C_\beta)>\delta>0
$$
for all $n\ge n_0$. This implies that
\begin{equation}\notag
\begin{split}
|x_{n+1}|^2&\le |x_n|^2-|y_n|^2=|x_n|^2-{\rm dist\,}(x_n,C_{\alpha(n)})^2 \\
&\le |x_n|^2-t_n^2{\rm dist\,}(x_n,C_\beta)^2\le |x_n|^2-t_n^2\delta^2 \\
&\le \dots \le |x_{n_0}|^2-\delta^2\sum_{\nu=n_0}^n t_\nu^2.
\end{split}
\end{equation}
Since  (\ref{teml_cond}) implies that $\sum t_\nu^2=\infty$, this  contradicts   $|x_n|\downarrow R>0$.
\end{proof}

Assume  $C_1, \dots, C_K$ are closed linear subspaces of $H$. Assume  $\{\alpha(n)\}$ is a quasi-periodic sequence of the indices $1,\dots, K$. This means that there is a constant $M\in {\mathbb N}$ so that for every interval $I$ of length $M$ the set $\{\alpha(n):\, n\in I\}$ contains all of the indices:
$$
\{\alpha(n):\, n\in I\}=\{1,\dots, K\}.
$$
Then the sequence $x_{n+1}=P_{\alpha(n)}x_n$ of projections converges in norm \cite{N}, \cite{Ha}, \cite{S}. Already for two closed and convex sets this is not true, as the example of Hundal exhibits \cite{H}, \cite{K}, \cite{MR}. Theorem \ref{theorem1} implies easily,
that as soon as the closed and convex sets $C_1, \dots, C_K$ are also  quasi-symmetric, convergence occurs.
For {\em symmetric} in place of {\em quasi-symmetric} this was established in \cite{BBR}.

\begin{corollary}\label{almostperiodic}
Assume  $C_1, \dots, C_K$ are finitely many closed, convex and  quasi-symmetric subsets of $H$ with a nonempty intersection $C=\bigcap_1^K C_j$. Assume  $\{\alpha(n)\}$ is a quasi-periodic sequence of the indices $1,\dots, K$.
Then the sequence $x_{n+1}=P_{\alpha(n)}x_n$ of nearest point projections converges in norm to a point in $C$ for any starting point $x_0\in H$.
\end{corollary}
\begin{proof}
The given sets are  quasi-symmetric and there are only finitely many of them, so the family is uniformly quasi-symmetric.
We will show there are weakness parameters $t_n\in [0,1]$ satifying $\sum t_n/n=\infty$ so that the sequence $\{x_n\}$ corresponds to a sequence of remote projections with these parameters. Hence according to Theorem~\ref{theorem1} the sequence $\{x_n\}$ converges in norm.

We choose  $\beta(n)\in \{1,\dots,\,K\}$ and define $b_n>0$ and $t_n\in [0,1]$ as follows:
\begin{equation}\notag
\begin{split}
{\rm dist\,}(x_n,C_{\beta(n)})&=\max_{k} {\rm dist\,}(x_n,C_k)=b_n, \\
t_n&=|x_{n+1}-x_n|/b_n.
\end{split}
\end{equation}
We will prove that for each interval $I$ of length $M$ there is an $n\in I$ so that $t_n\geq  1/(6M)$ and hence $\sum t_n/n=\infty$; here $M$ is the constant of the quasi-periodicity of $\{\alpha_n\}$.

Assume for a contradiction that there is $m\in {\mathbb N}$  so that $t_{m+j}<1/(6M)$ for all $j\in \{0,\dots, M\}$. We will show that then $\beta(m)\notin\{\alpha(m+j):\, j= 0,\dots, M\}$ contradicting the sequence $\{\alpha(n)\}$ of indices being quasi-periodic with constant $M$.
Indeed, by the triangle inequality,
\begin{equation}\notag
\begin{split}
b_{n+1}&=|x_{n+1}-P_{\beta(n+1)}x_{n+1}| \\
&\leq |x_{n+1}-x_n|+|x_n-P_{\beta(n+1)}x_n|+|P_{\beta(n+1)}(x_n-x_{n+1})| \\
&\leq \dfrac{b_n}{6M}+b_n+\dfrac{b_n}{6M}=\left(1+\dfrac1{3M}\right)b_n,
\end{split}
\end{equation}
for $m\leq n\leq m+M-1$. By induction, for any $1\leq k\leq M$,
$$
b_{m+k}\leq \left(1+\dfrac1{3M}\right)^Mb_m\leq 2b_m.
$$
Again by triangle inequalities
$$
|x_m-x_{m+k}|\leq \dfrac{k}{6M}2b_m<b_m,
$$
hence $x_{m+k}\notin C_{\beta(m)}$.
\end{proof}

Assume $\{C_\alpha\}_{\alpha\in \Omega}$ is a family of closed  subspaces in  $H$ and that $|\Omega|\geq 2$ is at most countable.  Assume that in a   sequence $\{\alpha(n)\}\subset \Omega$ each element of $\Omega$ appears infinitely many times. The sequence of consecutive projections  $x_{n+1}=P_{\alpha(n)} x_n$, $n=0,1,2,\dots$, $x_0\in H$, generated by $\alpha$ does not have to converge in general. However, if the norm limit (or even just the weak limit) of the sequence   exists, then it is equal to $P_Cx_0$, where $C=\bigcap_\alpha C_\alpha$.

Already for three closed subspaces $C_1, C_2, C_3$ and the sequence of remote projections (\ref{rem}) we can choose some of the weakness parameters $t_n\in [0,1]$ so small that the subspace $C_3$ can be completely avoided. This causes   (\ref{rem}) to converge to $P_{C_1\cap C_2}x_0$ which can be arranged to differ from $P_Cx_0$.

Already for two closed convex sets things can go awry even for the remotest projections, that is, if in (\ref{rem}) we set $t_n=1$ for all $n\in {\mathbb N}$.

\begin{example}\label{not_to_P_C_x_0}
In the Euclidean plane $H={\mathbb R}^2$ there are two closed, convex and  symmetric   sets $C_1$ and $C_2$,  and a starting point $x_0$ so that  the limit point of the remotest projections is not equal to $P_{C_1\cap C_2}x_0$.
\end{example}
\begin{proof}
In the coordinate representation $(s,t)$ of vectors in ${\mathbb R}^2$, we set
$$
C_1=\{s=0\}, \qquad C_2=\{s-2\le t \le s+2\}.
$$
The line $C_1$ and the stripe $C_2$ are both symmetric with respect to 0, and their intersection is the segment $C=\{(0,t): t\in [-2,2]\}$.
For the starting point $x_0=(-4,4)$, we have
$$
{\rm dist\,}(x_0, C_2)=3\sqrt2>4={\rm dist\,}(x_0,C_1).
$$
Hence  $x_1=P_2x_0=(-1,1)$ and
$x_2=P_1x_1=(0,1)\in C$, whereas $P_Cx_0=(0,2)$.
\end{proof}

Note that for finitely many closed convex sets there are special projection algorithms converging to the projection of the starting point onto their intersection~\cite[Ch.\, 30]{BC}.

\section{Symmetry conditions}\label{sym}

Dye and Reich~\cite{DR} introduced the following property of weakened symmetry.

\begin{definition}
Let $C$ be a closed convex set in a Banach space $X$. The origin is a {\em weak internal point} (shortly WIP) of $C$ if
\begin{equation}\label{wip}
      \forall x\in C \,\, \exists\, \delta=\delta(x)>0:\,    -\delta x \in C.
\end{equation}
Moreover, we say that $a\in X$ is a WIP-point of  $C$ if  the origin is a WIP-point of   the set  $(C-a)$.
\end{definition}

Clearly, the origin is a WIP-point of $C$ if and only if it is a WIP-point of $C_1=C\cap B(0,1)$. It is also easy to see that
the origin is a WIP-point of  a quasi-symmetric set: the condition (i) of Definition~\ref{quasym} seems to be stronger than (\ref{wip}). Surprisingly, the converse   is also true: $\delta(x)$ in (\ref{wip}) can be chosen independently of $x$ lying in the unit ball, say. Closed convex sets in Banach spaces cannot be too asymmetric.

\begin{remark}\label{quasisym}
Let $C$ be a closed  and convex set in a Banach space $X$. A point $a\in X$ is a weak internal point of $C$ if and only if $C$ is  quasi-symmetric with respect to $a$.
\end{remark}
\begin{proof} We show only the less obvious implication.
We assume that  $a=0$ and that  $C=C\cap B(0,1)$. We take the maximal possible $\delta$ which works in  (\ref{wip}):
for every $0\neq x\in C$ there exists $\delta(x)>0$ so that
\begin{enumerate}
\item[(i)] $-\delta(x)x\in C$;
\item[(ii)] if $\eta>\delta(x)$, then $-\eta x\notin C$.
\end{enumerate}
We claim that $\inf_{x\in C} \delta(x)>0$ and give an elementary proof of this fact first.
If not, then there are non-zero elements $e_n\in C$ having $\delta(e_n)<1/3^n$, $n\in {\mathbb N}$. Then
$$
e=\sum_{n=1}^\infty \frac{e_n}{2^n}\in C
$$
and we may assume $e\not=0$; otherwise we take $(1-\varepsilon)e_1$ instead of $e_1$ for sufficiently small $\varepsilon>0$.
Then $-\delta e\in C$ for some $\delta>0$.
For a  fixed $k\in {\mathbb N}$ we observe that
$$
\frac{1}{1+\delta(1-1/2^k)}+\sum_{{\mathbb N}\ni n\not=k}\frac{1}{1+\delta(1-1/2^k)}\frac{\delta}{2^n}=1;
$$
all the summands  on the left-hand side  are positive. Consequently,
$$
\frac{1}{1+\delta(1-1/2^k)}(-\delta e)+\sum_{{\mathbb N}\ni n\not=k}\frac{1}{1+\delta(1-1/2^k)}\frac{\delta}{2^n}e_n\in C,
$$
that is,
$$
\frac{1}{1+\delta(1-1/2^k)}\left(-\delta \sum_{n=1}^\infty \frac{e_n}{2^n}+\sum_{{\mathbb N}\ni n\not=k}\frac{\delta}{2^n}e_n\right)=\frac{-\delta/2^k}{1+\delta(1-1/2^k)}e_k\in C.
$$
Hence,
$$
\frac{\delta/2^k}{1+\delta(1-1/2^k)}\leq \delta(e_k)<\frac{1}{3^k}.
$$
The last inequality implies that
$$
\frac{\delta}{2^k}<\frac{1+\delta}{3^k}-\frac{\delta}{6^k},
$$
which is impossible for  large $k$'s; how large exactly depends on $\delta$.

Here is a  ``Baire category'' proof of the fact that  $\inf_{x\in C} \delta(x)>0$ due to V.I.\,Bogachev.  According to~\cite[Proposition 2.5.1]{BS} both sets $C\cap(-C)$ and ${\rm conv\,} (C\cup (-C)) $  generate norms on ${\rm span\,} C$ in which ${\rm span\,} C$ is a Banach space. The open mapping theorem implies that the two  norms are equivalent, hence  the above infimum  is  positive.

\end{proof}

Next we exhibit that the {\em uniform} quasi-symmetry assumption on the sets $C_{\alpha}$ in Theorem~\ref{theorem1} is essential.
In \cite{H} and \cite{K}, an example of  a closed convex cone  $C$ with the vertex at the origin was constructed so that iterating the nearest point projection between $C$ and a hyperplane $D$  converges weakly but not in norm for a starting point $x_0\in D$. In the example the  hyperplane $D=e^\perp$ for an $0\neq e\in H=\ell_2$ and the set $C$ is the epigraph in $\ell_2=D+ {\rm span\,}\{e\}$
of a suitably chosen nonnegative convex sublinear function defined on $D$.
Those familiar with the example readily ``see'', that the family of closed convex sets consisting of $D$ and $C-c_ne$ for some suitable $c_n\searrow 0$ consists of quasi-symmetric sets for which the remote projections algorithm starting at $x_0$  closely traces the iterates of nearest points projections of $x_0$ between $C$ and $D$. Consequently it converges weakly but not in norm.
Rather than writing this up rigorously we give here a construction which is easier to present.

\begin{example}\label{nunifsym}
In any infinite dimensional Hilbert space $H$,
there exists a countable family of closed, convex  and quasi-symmetric sets    so that the sequence  of remotest projections on this family does  not converge in norm for a certain starting point.
\end{example}
\begin{proof} We assume that $H$ is separable as if it is not,  then we build the example in a closed separable infinite dimensional subspace of $H$.
Also, we construct a family of sets and a point in their intersection so that each set in the family is quasi-symmetric with respect to this point. To center at the origin, we translate, if need be.

We use as a building stone an example constructed in~\cite{B}; we first recall its relevant properties.

Let $\{e, e_k: k\in {\mathbb N}\}$ be an orthonormal basis of $H$.
For each $k\in {\mathbb N}$, we choose  vectors $v_1^k,\dots, v_{n_k}^k\in {\rm span\,}\{e_k, e_{k+1}\}$ as in \cite{B}. Their number  $n_k$ increases in a particular  way, the norms $|v_n^k|$ decrease  in a particular way. Their only property relevant here   are  their directions:
\begin{equation}\label{arg_v_n^k}
  \arg v_n^k=-\frac{\pi}{2}+\frac{\pi n}{n_k}, \qquad k\in {\mathbb N},  n=1,\dots,n_k;
\end{equation}
here the polar  angle $\arg$ in the plane ${\rm span\,}\{e_k, e_{k+1}\}$  is measured  from the positive direction of $e_k$.

The diverging greedy algorithm with respect to the dictionary containing  $\pm e$ and all vectors $(e+v_n^k)/|e+v_n^k|$  which is constructed in~\cite{B}  can be  interpreted as the process of remotest projections onto the family of closed convex sets consisting of the hyperplane $D=e^\perp$ and the half-spaces
$$
C_{n,k}=\{y\in H: \langle y, e+v_n^k\rangle \le 0\}.
$$
Starting with $x_0=e_1$, the remotest projections algorithm generates   $x_{m+1}=P_{C_{n,k}}x_m$ for even $m$ ($k$ and $n$ depending on $m$) and $x_{m+1}=P_{D}x_m$ for odd $m$. For all $m$ and $k$, the inequalities $\langle x_m,e_k\rangle\ge 0$ and $\langle x_m,e\rangle\le 0$ hold. The sequence $\{x_m\}$ converges  to 0 weakly but not in norm; for more details see~\cite{B}.

The hyperplane  $D$ and all the half-spaces $C_{n,k}$ are quasi-symmetric with respect to any point
$$
a\in D\cap \left(\cap_{n,k} C_{n,k}^\circ\right),
$$
where $C_{n,k}^\circ$ denotes the interior of the half-space $C_{n,k}$.
We define the coordinates of such a point
$$
a=(0,a^1,a^2,\dots)
$$
with respect to the  basis $\{e, e_k: k\in {\mathbb N}\}$, recursively:
$$
a^1=-1, a^{k+1}=a^k \tan \frac{\pi}{4n_k}.
$$
Clearly, $a\in \ell_2$, $a\in D$, and (\ref{arg_v_n^k}) implies that
$$
\langle a, e+v_n^k\rangle =\langle a, v_n^k\rangle =\langle a^ke_k+a^{k+1}e_{k+1}, v_n^k\rangle <0,
$$
since
$$
\arg (a^ke_k+a^{k+1}e_{k+1})=-\pi+\frac{\pi}{4n_k}
$$
in the plane ${\rm span\,}\{e_k, e_{k+1}\}$. Hence, $a\in C_{n,k}^\circ$ for each $k$ and $n$.
\end{proof}

If the interior of the  intersection of a family of closed convex sets is non-empty, then, clearly, the family is uniformly quasi-symmetric. In such a case, any sequence of projections onto these sets converges. For remote projections we even give  an estimate of the rate.

\begin{remark}\label{ball}
Let each closed convex set $C_\alpha$ contain the ball $B(a,r)$, $a\in H$, $r>0$.
\begin{enumerate}
  \item[(a)] The sequence (\ref{rem}) of remote projections converges in norm for each starting element $x_0\in H$ and for any sequence $\{t_n\}$. In particular,    random projections converge.
  \item[(b)] If, moreover, $\sum t_n^2=\infty$, then the limit point $w$ belongs to $ \bigcap_{\alpha\in \Omega} C_\alpha$, and   the rate of convergence is estimated by
      \begin{equation}\label{rate}
        |x_n-w|\le 2 |x_0-a| \prod_{k=0}^{n-1}\left(1-\frac{t_k^2r^2}{|x_0-a|^2}\right)^{1/2}.
      \end{equation}
\end{enumerate}
\end{remark}

The statement (b) clarifies a result from~\cite{GBR}. There the convergence to a  point in the intersection  was shown  under the condition $\sup_\alpha {\rm dist\,}(x_n,C_\alpha)\to 0$ as $n\to\infty$. Also, an exponential rate of convergence was established for remotest projections ($t_n\equiv 1$) with an estimate similar to ours.

\begin{proof} (a) We assume $a=0$ and use the notations  $y_n=x_n-x_{n+1}$, $\varepsilon_n=\pi/2-\angle 0x_{n+1}x_n$, and also several inequalities from the proof of Theorem \ref{theorem1}.

In view of (\ref{ext_normal}), we have $\langle y_n, z-x_{n+1}\rangle \le 0$ for any $z\in B(0,r)$. Consequently,
$$
|y_n||x_{n+1}|\sin \varepsilon_n =\langle  y_n,x_{n+1}\rangle\ge \sup _{z\in B(0,r)}\langle y_n,z\rangle=r|y_n|,
$$
so that
$$
\sin \varepsilon_n\ge \frac{r}{|x_{n+1}|}\ge \frac{r}{|x_0|}.
$$
This estimate together with (\ref{sin}) yields $\sum |y_n|<\infty$, meaning that $x_n$ converge in norm.

(b) To prove that the limit point $w$ belongs to $C=\bigcap_{\alpha\in \Omega} C_\alpha$ in case $\sum t_n^2=\infty$, one can use the same arguments as in part 6 of the proof of Theorem \ref{theorem1}.

Now we proceed to prove (\ref{rate}). Note that for any $n\in {\mathbb N}$ we have
\begin{equation}\label{est_via_dist}
  |x_n-w|\le 2\,{\rm dist\,} (x_n,C),
\end{equation}
otherwise
$$
|x_m-y|<\frac{|x_m-w|}{2}
$$
for some $y\in C$ and $m\in {\mathbb N}$, so that
$$
|x_n-y|\to |w-y|\ge |x_m-w|-|x_m-y|>|x_m-y|,
$$
which contradicts the fact that the sequence $\{|x_n-y|\}$ is decreasing.

Let $n\in {\mathbb N}$ and
$$
d_n=\sup_\alpha {\rm dist\,}(x_n,C_\alpha).
$$
The  ball $B(x_n,d_n)$ contains a point $p_\alpha\in C_\alpha$ for each $\alpha$. Since the point
$$
u_n=\frac{d_n}{d_n+r}a+\frac{r}{d_n+r}x_n
$$
belongs to ${\rm conv\,}\{p,B(a,r)\}$ for each $p\in B(x_n,d_n)$, we get $u_n\in C$, so that
$$
{\rm dist\,}(x_n,C)\le |x_n-u_n|=|x_n-a|\cdot \frac{d_n}{d_n+r}.
$$
Consequently,
$$
d_n\ge \frac{r\,{\rm dist\,}(x_n,C)}{|x_n-a|-{\rm dist\,}(x_n,C)}\ge \frac{r\,{\rm dist\,}(x_n,C)}{|x_0-a|}.
$$

Let $P_Cx_n=b$. Since $b\in C_{\alpha(n)}$, the angle $\angle x_nx_{n+1}b$ is not less than $ \pi/2$, so that
$$
{\rm dist\,}(x_{n+1},C)^2\le |x_{n+1}-b|^2\le |x_n-b|^2-|x_n-x_{n+1}|^2
$$
$$
\le {\rm dist\,}(x_n,C)^2-t_n^2d_n^2\le {\rm dist\,}(x_n,C)^2\left(1-\frac{t_n^2r^2}{|x_0-a|^2}\right).
$$
Hence,
$$
{\rm dist\,}(x_{n+1},C)\le {\rm dist\,}(x_0,C)\prod_{k=0}^n \left(1-\frac{t_k^2r^2}{|x_0-a|^2}\right)^{1/2},
$$
which together with (\ref{est_via_dist}) gives (\ref{rate}).

\end{proof}

Assume that unlike the assumption in Remark~\ref{ball} we deal with a family of slim sets: all $C_\alpha$ are hyperplanes $g_\alpha^\perp$. Then 
 remote projections implement the Weak Greedy Algorithm with respect to the dictionary $D=\{\pm g_\alpha: \alpha\in \Omega\}$ and there are  estimates of the   rate  of convergence for starting elements from $\overline{{\rm conv\,}} D$~\cite[Ch.\,2]{T}.
We wonder if any such estimates can be shown for a class of starting elements in the general setting of Theorem~\ref{theorem1}.

\section{Weak convergence}\label{weak}

Bregman~\cite{Bre} proved that for any family of general (non-symmetric)  closed convex sets with nonempty intersection the remotest projections (\ref{rem}) with $t_n\equiv 1$ always converge weakly. He assumed that $\max_{\alpha} {\rm dist\,}(x,C_\alpha)$ is attained for each $x\in H$. It is quite natural to generalize this result to remote projections by slightly changing his arguments.

\begin{theorem}\label{theorem2}
Assume $\{C_\alpha\}$ is a family of closed and convex sets in a Hilbert space  $H$ with a nonempty intersection $C= \bigcap_{\alpha\in \Omega} C_\alpha$.
Let the sequence $\{t_n\}$ in $[0,1]$ satisfy the following condition: there are $\delta>0$ and $K\in{\mathbb N}$ so that  for any $n\in {\mathbb N}$ at least one of the values $t_n,\dots,t_{n+K}$ is greater than $\delta$.
Then the sequence (\ref{rem}) of remote projections converges weakly to some point of $C$ for any starting element $x_0\in H$.
\end{theorem}
\begin{proof}
Take any $n\in {\mathbb N}$ and $k\in \{n,\dots,n+K\}$ so that $t_k>\delta$. We use the notation $y_\nu=x_\nu-x_{\nu+1}$.  For any $\alpha\in \Omega$, we have
$$
{\rm dist\,}(x_n,C_\alpha)\le |x_n-x_k|+{\rm dist\,}(x_k,C_\alpha)\le \sum_{i=n}^{k-1}|y_i|+\frac{{\rm dist\,}(x_k,C_{\alpha(k)})}{t_k}
$$
$$
\le \sqrt{K}\left(\sum_{i=n}^{n+K-1}|y_i|^2\right)^{1/2}+\frac{|y_k|}{\delta}\to 0 \qquad (n\to \infty),
$$
since $\sum |y_\nu|^2<\infty$ in view of (\ref{y_n}). Consequently, each partial weak limit $w$ of the  $x_n$'s belongs to $C$. In its turn, this implies 
that the whole sequence $|x_n-w|$ is decreasing, since each $P_{\alpha}$ is a 1-Lipschitz retraction onto $C_{\alpha}$.

A partial weak limit does exist, so we have to prove its uniqueness. Let $v$ and $w$ be two partial weak limits, so that $x_{n_i}$ converge weakly to $v$ and $x_{m_j}$ converge weakly to $w$. The numbers
$$
d_n:=|x_n-v|^2-|x_n-w|^2=2\langle v-x_n,v-w\rangle -|v-w|^2=2\langle x_n-w,w-v\rangle +|w-v|^2
$$
tend to a single limit, as we have just mentioned. On the other hand, $d_{n_i}\to -|v-w|^2$ and $d_{m_j}\to |w-v|^2$. Hence, $v=w$.
\end{proof}

The following result was established by Dye and Reich in \cite{DR}. 

\begin{corollary}\label{walmostperiodic}
Assume  $C_1, \dots, C_K$ are finitely many closed and  convex   subsets of $H$ with a nonempty intersection $C=\bigcap_1^K C_j$. Assume  $\{\alpha(n)\}$ is a quasi-periodic sequence of the indices $1,\dots, K$.
Then the sequence $x_{n+1}=P_{\alpha(n)}x_n$ of nearest point projections converges weakly to a point in $C$ for any starting point $x_0\in H$.
\end{corollary}
\begin{proof}
The sequence  $\{\alpha(n)\}$ is  quasi-periodic, which means that there is a constant $M\in {\mathbb N}$ so that for every interval $I$ of length $M$ the set $\{\alpha(n):\, n\in I\}$ contains all of the indices $1,\dots, K$.   As in the proof of Corollary~\ref{almostperiodic} we choose  $\beta(n)\in \{1,\dots,\,K\}$ and define $b_n>0$ and  weakness parameters $t_n\in [0,1]$ as follows:
\begin{equation}\notag
\begin{split}
{\rm dist\,}(x_n,C_{\beta(n)})&=\max_{k} {\rm dist\,}(x_n,C_k)=b_n, \\
t_n&=|x_{n+1}-x_n|/b_n.
\end{split}
\end{equation}
Then  for each interval $I$ of length $M$ there is an $n\in I$ so that $t_n\geq  1/(6M)$, as shown in the proof of Corollary~\ref{almostperiodic}. According to Theorem~\ref{theorem2} the sequence $\{x_n\}$ converges weakly to a point of $C$.
\end{proof}

We do not know if the  condition on the sequence $\{t_n\}$ in Theorem~\ref{theorem2} is necessary for the weak convergence of remote projections. It  is much stronger than  the condition (\ref{teml_cond}) implying  the norm convergence of remote projections in the uniformly quasi-symmetric case. We do not know of an equivalent condition for the weak convergence   in the uniformly quasi-symmetric  case either.
We give, however,  criteria for remote projections to have a partial weak limit in the intersection of the sets considered.

\begin{theorem}\label{theorem3} For a sequence $\{t_n\}_{n=0}^\infty \subset [0,1]$, the following statements are equivalent:
\begin{enumerate}
\item[(i)] $\sum t_n^2=\infty$;
\item[(ii)] the sequence $\{x_n\}$ of remote projections (\ref{rem}) with parameters $t_n$ has a partial weak limit in $\bigcap_{\alpha\in \Omega} C_\alpha$  for any starting element $x_0\in H$ and any family $\{C_\alpha\}_{\alpha\in \Omega}$ of closed and convex sets in $H$ with nonempty intersection;
\item[(iii)] the residuals $\{x_n\}$ in the Weak Greedy Algorithm (\ref{greal}) with parameters $t_n$ have a partial weak limit 0 for any starting element $x_0\in H$ and any dictionary $D\subset S(H)$.
\end{enumerate}
\end{theorem}
\begin{proof}
(i) $\Rightarrow$ (ii). Let $\sum t_n^2=\infty$, $0\in C=\bigcap_{\alpha\in \Omega} C_\alpha$ and $\{x_n\}$ be the sequence (\ref{rem}). We denote $y_n=x_n-x_{n+1}$ again. Since $\sum |y_n|^2<\infty$ in view of (\ref{y_n}), we can choose a subsequence $\Lambda\subset {\mathbb N}$ with the property
$|y_n|/t_n\to 0$, $n\in \Lambda$, $n\to \infty$. Taking any $\alpha\in \Omega$ and $n\in \Lambda$, we get
$$
{\rm dist\,}(x_n,C_\alpha)\le \frac{{\rm dist\,}(x_n,C_{\alpha(n)})}{t_n}=\frac{|y_n|}{t_n}\to 0, \qquad n\to \infty.
$$
Consequently, any partial weak limit of $\{x_n\}_{n\in \Lambda}$ belongs to $C$. In fact, this partial weak limit in $C$ is unique, as we have seen in the proof of Theorem \ref{theorem2}. However, there may be other partial weak limits outside of $C$, as  Example~\ref{weak_div} below shows.

(ii) $\Rightarrow$ (iii). This is obvious, since the WGA is a particular case of remote projections onto a family of hyperplanes having unique common point 0.

(iii) $\Rightarrow$ (i). Let  $\sum t_n^2<\infty$. In what follows we  construct a countable family $\{C_n\}$ of one-dimensional subspaces  and a sequence (\ref{rem}) of remote projections onto this family with parameters $t_n$, which does not converge weakly and does not have 0 as a partial weak limit.  Remark~\ref{WGA} then supplies an example of the WGA  with parameters $t_n$ whose residuals do not have 0 as a partial weak limit.

We choose a sequence $\{\tau_n\}$ with the properties $\tau_n\ge t_n$ for all $n$, $\sum \tau_n^2<\infty$, $\sum \tau_n=\infty$, and fix $m$ so that
\begin{equation}\label{m}
\sum_{n=m}^\infty \tau_n^2< \frac{1}{4}.
\end{equation}

We fix a point  $s$  on the unit sphere of $H$ and  take  the spherical cap
$$
V(s)=\left\{v\in S(H): \langle v,s\rangle\ge \frac{\sqrt3}{2}\right\}.
$$
We choose two {\em opposite} points $a$ and  $b$ on the boundary of  the cap:   $\langle a,s\rangle=\langle b,s\rangle= \sqrt 3/2$. We also choose  a sequence $\{s_n\}_{n=m}^\infty\subset V(s)$ so that
$s_m=a$ and $\langle s_n,s_{n+1}\rangle=\sqrt{1-\tau_n^2}$ for all $n\ge m$. This sequence can be constructed inductively: the choice of each next $s_{n+1}\in V(s)$ is possible, since $\sqrt{1-\tau_n^2}\ge \sqrt3/2$  by (\ref{m}). Moreover, we have
$$
|s_n-s_{n+1}|=\sqrt{2-2\sqrt{1-\tau_n^2}}\ge \tau_n,
$$
so that $\{s_n\}_{n=m}^\infty$ may be made dense in $V(s)$,  since $\sum \tau_n=\infty$.

Denoting by $L(v):={\rm span\,} \{v\}$ the line spanned by a  vector $v\in S(H)$, we consider the family of lines
$$
L(a), L(b), L(s_n), \qquad n=m,m+1,\dots,
$$
and the following sequence of remote projections onto this family of lines with starting element $x_0=s$. The projections $x_1,\dots,x_m$ alternately lie on the lines $L(a)$ and $L(b)$ so that they are remotest and satisfy the inequalities needed for any given parameters $t_0,\dots,t_{m-1}$. We choose the first projection $x_1$ lying either on $L(a)$ or $L(b)$ so that $x_m\in L(a)$, depending on the parity of $m$. As for $n\ge m$, we set
$x_{n+1}$ to be the projection of $x_n$ onto $L(s_{n+1})$:
\begin{equation*}
\begin{split}
    {\rm dist\,}(x_n,L(s_{n+1}))=& |x_n|\sin \angle (s_n0s_{n+1})\\
    =&|x_n| \tau_n\ge |x_n| t_n \ge t_n \sup_{v\in V(s)}{\rm dist\,}(x_n,L(v))
\end{split}
\end{equation*}

Clearly, the sequence $\{x_n\}$ is contained in the cone $\{ \lambda v:\, v\in V(s),\  \lambda>0\} $.
For any $n> m$, we have
$$
|x_n|^2=|x_m|^2-\sum_{k=m}^{n-1}|x_k|^2\tau_k^2\ge |x_m|^2\left(1-\sum_{k=m}^\infty \tau_k^2\right)\ge \frac{3}{4}|x_m|^2.
$$
This means that $|x_n|\to r>0$, and the set of all partial weak limits of the sequence $\{x_n\}$ is the closed convex hull of the cap $rV(s)$.

\end{proof}

The following Example shows that the conditions on the sequence $\{t_n\}$ in Theorem~\ref{theorem2} cannot be replaced by $\liminf_{n\to \infty}t_n>0$ and that in Theorem~\ref{theorem3} one cannot claim the uniqueness of the weak limit.

\begin{example}\label{weak_div}
Let $H$ be an infinite dimensional Hilbert space. 
Then 
there exists a countable family  of closed convex sets in $H$ with non-empty intersection    and a  sequence (\ref{rem}) of remote projections on this family  which does not converge weakly and its weakness parameters satisfy $\liminf_{n\to \infty}t_n>0$.
\end{example}
\begin{proof}
1. We use the following local construction.
\begin{lemmaX}\label{lemmaA}~\cite[Section 2.2]{K}
Let $a,b,c\in H$ be such that $|a|=|b|\neq 0$ and  $0\neq c\in \{a, b\}^{\perp}$. For every $\varepsilon>0$ there exists a convex closed cone $C=C(a,b,c,\varepsilon)\subset {\rm span\,}\{a,b,c\}$ with vertex 0 so that alternating projections between $C$ and the plane $D= {\rm span\,}\{a,b\}$ move the point $a$ close to the point $b$:
\begin{equation}\label{alter_hund}
\left|(P_D P_C)^ma-b\right|<\varepsilon
\end{equation}
for some $m=m(a,b,c,\varepsilon)$.
\end{lemmaX}

The cone from Lemma~\ref{lemmaA} also satisfies  
\begin{equation}\label{dist(a,C)}
  {\rm dist\,}(a,C)\le \sqrt{2|a|\varepsilon},
\end{equation}
since otherwise
$$
\left|(P_D P_C)^ma\right|\le \left|P_Ca\right|=\sqrt{|a|^2-{\rm dist\,}(a,C)^2}<|a|-\varepsilon=|b|-\varepsilon,
$$
which contradicts (\ref{alter_hund}). Similarly,
$$
\left|P_C(P_D P_C)^{m-1}a-(P_D P_C)^ma\right|={\rm dist\,}(P_C(P_D P_C)^{m-1}a,D)\le \sqrt{2|a|\varepsilon},
$$
and hence
\begin{equation}\label{dist(b,C)}
\begin{split}
{\rm dist\,}(b,C)&\le \left|b-P_C(P_D P_C)^{m-1}a\right|\le \left|b-(P_D P_C)^{m}a\right| \\
&+\left|P_C(P_D P_C)^{m-1}a-(P_D P_C)^ma\right|\le \varepsilon+\sqrt{2|a|\varepsilon}.
\end{split}
\end{equation}

2. We may assume that $H$ is separable and fix an orthonormal basis 
 $\{u,\, v,\, e_k:\, k\in {\mathbb N}\}$  of $H$. We choose a decreasing   sequence $\varepsilon_n\searrow 0$  so that
$$
\sum_{k=1}^\infty \sqrt{\varepsilon_k}<\frac 1{10}.
$$
We set
\begin{equation}\notag
\begin{split}
D&=v^\perp=\overline{{\rm span\,}}\{u,e_k: k\in {\mathbb N}\},\\
C_1&=C(e_1,u,v,\varepsilon_1)+\overline{{\rm span\,}}\{e_n: n\in {\mathbb N}, n\not=1\},\\
m_1&=m(e_1,u,v,\varepsilon_1),\\
C_2&=C(u,e_2,v,\varepsilon_2)+\overline{{\rm span\,}}\{e_n: n\in {\mathbb N}, n\not=2\},\\
m_2&=m(u,e_2,v,\varepsilon_2),\\
&\dots\\
C_{2k-1}&=C(e_k,u,v,\varepsilon_{2k-1})+\overline{{\rm span\,}}\{e_n: n\in {\mathbb N}, n\not=k\},\\
m_{2k-1}&=m(e_k,u,v,\varepsilon_{2k-1}),\\
C_{2k}&=C(u,e_{k+1},v,\varepsilon_{2k})+\overline{{\rm span\,}}\{e_n: n\in {\mathbb N}, n\not=k+1\},\\
m_{2k}&=m(u,e_{k+1},v,\varepsilon_{2k}),\\
&\dots.
\end{split}
\end{equation}
Clearly,  (\ref{alter_hund}) works  for the extended cones $C_{2k-1}$ and  $C_{2k}$ as well:
\begin{equation}\label{alter_hund_trick_odd}
\left|(P_D P_{C_{2k-1}})^{m_{2k-1}}e_k-u\right|<\varepsilon_{2k-1},
\end{equation}
\begin{equation}\label{alter_hund_trick_even}
\left|(P_D P_{C_{2k}})^{m_{2k}}u-e_{k+1}\right|<\varepsilon_{2k}.
\end{equation}

3. We have $e_k\in D$ for all $k$ and  $e_k\in C_n$ for $n\not= 2k-2,2k-1$ by construction, hence also
\begin{equation}\notag
\begin{split}
{\rm dist\,}(e_k,C_{2k-1})&<\sqrt{2\varepsilon_{2k-1}},\\
{\rm dist\,}(e_k,C_{2k-2})&<\sqrt{2\varepsilon_{2k-2}}+\varepsilon_{2k-2}
\end{split}
\end{equation}
by (\ref{dist(a,C)}) and  (\ref{dist(b,C)}). This implies  for  $P$ being a  projection onto $C_n$ or $D$ that
\begin{equation}\label{kissing_point}
|e_k-Pe_k|< 3\sqrt{\varepsilon_{2k-2}}, \qquad k=2,3,\dots.
\end{equation}

4. Now we define the required sequence of remote projections on the family $\{D,\, C_n:\, n\in {\mathbb N}\}$.

We start with $x_0=e_1$ and make $m_1$ alternating projections on $C_1$ and $D$:
$$
y_1=(P_DP_{C_1})^{m_1}e_1.
$$
Then we make $m_2$ alternating projections on $C_2$ and $D$:
$$
y_2=(P_DP_{C_2})^{m_2}y_1.
$$
Then we make the projection $P_2$ on one of the sets $C_n$ so that
$$
|y_2-P_2y_2|\ge \frac{1}{2}\sup_n {\rm dist\,}(y_2,C_n),
$$
and we set
$$
z_2=P_2y_2.
$$

We proceed by induction: having defined $y_1,y_2,\dots,y_{2k-2}$ and\\ $z_2,z_4,\dots,z_{2k-2}$, we make $m_{2k-1}$ alternating projections on $C_{2k-1}$ and $D$:
$$
y_{2k-1}=(P_DP_{C_{2k-1}})^{m_{2k-1}}z_{2k-2},
$$
then $m_{2k}$ alternating projections on $C_{2k}$ and $D$:
$$
y_{2k}=(P_DP_{C_{2k}})^{m_{2k}}y_{2k-1},
$$
and then one projection $P_{2k}$ on one of the sets $C_n$ so that
\begin{equation}\label{kissing_projection}
  |y_{2k}-P_{2k}y_{2k}|\ge \frac{1}{2}\sup_n {\rm dist\,}(y_{2k},C_n),
\end{equation}
and we set
$$
z_{2k}=P_{2k}y_{2k}.
$$

For this sequence of projections, containing subsequences $\{y_k\}$ and $\{z_{2k}\}$, we have
$$
\liminf_{n\to \infty} t_n\ge \frac{1}{2},
$$
since projections via $P_{2k}$ have $t_n\ge 1/2$ by (\ref{kissing_projection}).

5. At last we have to prove that the sequence  $\{y_k\}$ does not converge weakly, and hence the whole sequence of projections has no weak limit.

We have $|y_1-u|<\varepsilon_1$ by (\ref{alter_hund_trick_odd}) for $k=1$. Using (\ref{alter_hund_trick_even}) for $k=1$, we get
$$
|y_2-e_2| \le \left|(P_DP_{C_2})^{m_2}(y_1-u)\right|+ \left|(P_DP_{C_2})^{m_2}u-e_2\right|<\varepsilon_1+\varepsilon_2,
$$
which together with (\ref{kissing_point}) implies that
$$
|z_2-e_2|=|P_2y_2-e_2|\le |P_2(y_2-e_2)|+|P_2e_2-e_2|<\varepsilon_1+\varepsilon_2+3\sqrt{\varepsilon_2}.
$$

In the same way, by induction on $k$, we get
$$
|y_{2k-1}-u|<\sum_{\nu=1}^{2k-1} \varepsilon_\nu+3\sum_{\nu=1}^{k-1} \sqrt{\varepsilon_{2\nu}},
$$
$$
|y_{2k}-e_{k+1}|<\sum_{\nu=1}^{2k} \varepsilon_\nu+3\sum_{\nu=1}^{k-1} \sqrt{\varepsilon_{2\nu}},
$$
$$
|z_{2k}-e_{k+1}|<\sum_{\nu=1}^{2k} \varepsilon_\nu+3\sum_{\nu=1}^{k} \sqrt{\varepsilon_{2\nu}}.
$$

Consequently,
\begin{equation}\notag
\begin{split}
|y_{2k-1}-u|<0.4 &\Rightarrow \langle y_{2k-1},u\rangle >0.6\\
|y_{2k}-e_{k+1}|<0.4 & \Rightarrow \langle y_{2k},u\rangle <0.4
\end{split}
\end{equation}
and the sequence  $\{y_n\}$ does not converge weakly.
\end{proof}

The following statement is  a parallel to  Remark~\ref{ball}~(a). We consider here weak convergence instead of  norm convergence. The intersection of the sets is not contained in any affine hyperplane in place of having non-empty interior.

\begin{remark}\label{spanning set}
Let $\{C_{\alpha}\}$ be a family of closed convex subsets of a Hilbert space $H$. Assume that the affine hull of the intersection $C=\bigcap C_\alpha$
is dense in $H$.
Then the sequence (\ref{rem}) of remote projections converges weakly for each starting element $x_0\in H$ and for any  sequence $\{t_n\}$ of weakness parameters.
In particular, random projections converge weakly in this case.  
\end{remark}
 \begin{proof} Fix a point $a\in C$. Then $\overline{{\rm span\,}}\{v-a: v\in C\}=H$.
 For each $v\in C$ and any $n\in {\mathbb N}$, we have $v\in C_{\alpha(n)}$, hence $\angle vx_{n+1}x_n\ge \pi/2$, so that $|x_n-v|\ge |x_{n+1}-v|$, and the  decreasing sequence
\begin{equation}\label{x_n-v}
  |x_n-v|^2=  |x_n-a|^2-2\langle x_n-a, v-a \rangle    +|v-a|^2
\end{equation}
has a limit. In particular, the sequence $\{|x_n-a|^2\}$ has a limit, which together with (\ref{x_n-v}) implies that the sequence of  scalar products
$$
\langle x_n-a, v-a \rangle
$$
has a limit as well. The sequence  $\{x_n-a\}$ is  bounded and the set  ${\rm span\,} \{v-a: v\in C\}$ is dense in $H$, hence the sequence $\{x_n-a\}$ converges weakly, and so does the sequence $\{x_n\}$.
\end{proof}

Dye and Reich~\cite{DR} proved weak convergence of random projections on a finite family of closed convex sets that are all WIP sets with respect to their common point, see also~\cite{BK22}.  Such sets are uniformly quasi-symmetric with respect to this point by Remark~\ref{quasisym}.
We wonder if Theorem~\ref{theorem2} and Theorem~\ref{theorem3} can be clarified under the additional condition of uniform quasi-symmetry of the sets $C_\alpha$. We also note that the problem of weak convergence of random projections, that is, remote projections with arbitrary $t_n$'s, onto a finite family of closed convex sets having nonempty intersection is still open~\cite{BK22}.

\section*{Acknowledgements}
We thank S.\,Reich, V.N.\,Temlyakov, and V.I.\,Bogachev  for  fruitful discussions.

\end{document}